\documentclass[english]{amsart}

\usepackage{amssymb}
\usepackage{amsmath}
\usepackage{psfrag}
\usepackage[all]{xy}
\usepackage{graphicx}
\usepackage{epstopdf}
\usepackage{color}
\usepackage{microtype}
\usepackage{manfnt}

\xyoption{arc}

\newcommand{\Diff}{{\rm Diff}_+^r(F)}
\newcommand{\rs}{\backslash}
\newcommand{\SL}{SL(2,\R)}
\newcommand{\PSL}{PSL(2,\R)}
\newcommand{\Z}{\mathbb{Z}}
\newcommand{\R}{\mathbb{R}}
\newcommand{\C}{\mathbb{C}}
\newcommand{\N}{\mathbb{N}}
\newcommand{\Q}{\mathbb{Q}}
\newcommand{\Hy}{\mathbb{H}}
\newcommand{\F}{\mathcal{F}}

\newcommand{\A}{{\rm Aff}_+(\R)}

\newcommand{\M}{\mathcal{M}}

\newcommand{\scirc}{{\scriptstyle \circ}}

\newcommand{\matrice}[4]{\left(
\begin{array}{cc}
 #1 & #2 \\
 #3 & #4
\end{array}
\right) }
\newcommand{\vect}[2]{\big(
\raisebox{-0.8ex}{$\stackrel{^{\scriptstyle #1}}{\scriptstyle #2}$} \big) }
\newcommand{\oclosure}[1]{\overline{#1}^{\, \circ}}

\usepackage{accents}
\newlength{\dhatheight}

\newcounter{itheorem}
\newcounter{atheorem}

\newtheorem{theorem}{Theorem}[section]
\newtheorem{corollary}[theorem]{Corollary}
\newtheorem{lemma}[theorem]{Lemma}
\newtheorem{proposition}[theorem]{Proposition}
\newtheorem{aproposition}[atheorem]{Proposition}
\newtheorem{atheorem}[atheorem]{Theorem}
\newtheorem{itheorem}[itheorem]{Theorem}

\newtheorem{speciallemma}[theorem]{Classification Lemma}

\theoremstyle{definition}

   \newtheorem{example}[theorem]{Example}
    \newtheorem{examples}[theorem]{Examples}
\newtheorem{remark}[theorem]{Remark}

\newtheoremstyle{named}{}{}{\itshape}{}{\bfseries}{.}{.5em}{\thmnote{#3's }#1}
\theoremstyle{named}
\newtheorem*{namedtheorem}{Theorem}
\newtheorem*{namedquestion}{Question}
\newtheorem*{namedstheorem}{Structure Theorem}

\begin{document}

\title[Dynamics of the horocycle flow for homogeneous foliations ]{Remarks on the dynamics of the horocycle flow for homogeneous foliations by hyperbolic surfaces}
\author[F. Alcalde Cuesta]{Fernando Alcalde Cuesta} 
\address{GeoDynApp - ECSING group (Spain)}
\email{fernando.alcalde@usc.es}

\author[F. Dal'Bo]{Fran\c{c}oise Dal'Bo}
\address{Institut de Recherche Math\'ematiques de Rennes  \\ 
 Universit\'e de  Rennes 1 \\ F-35042 Rennes (France)}
\email{francoise.dalbo@univ-rennes1.fr}
\date{\today}

 


\dedicatory{Dedicated to Pierre Molino with admiration}

\begin{abstract}
This article is a first step towards the understanding of the dyna\-mics of the horocycle flow on foliated manifolds by hyperbolic surfaces. This is motivated by a question formulated by M. Mart\'{\i}nez and A. Verjovsky on the minimality of this flow assuming that the "natural" affine foliation is minimal too. We have tried to offer a simple presentation, which allows us to update and shed light on  the classical theorem proved by G. A. Hedlund in 1936 on the minimality  of the horocycle flow on compact hyperbolic surfaces. Firstly, we extend this result to the product of 
$\PSL$ and a Lie group $G$, which places us within the homogeneous framework investigated by M. Ratner. Since our purpose is to deal with non-homogeneous situations, we do not use Ratner's famous Orbit-Closure Theorem, but we give an elementary proof. 
We show that this special situation arises for homogeneous Riemannian and Lie foliations, reintroducing the foliation point of view. Examples and counter-examples take an important place in our work, in particular, the very instructive case of the solvable manifold $T^3_A$.
Our aim in writing this text is to offer to the reader an accessible introduction to a subject that was intensively studied in the algebraic setting, although there still are unsolved geometric problems.
\end{abstract}

\maketitle

\section{Introduction and motivation} \label{Smot}

In this paper, we start by focusing our attention on the following subgroups
$$
U = \{ \ \matrice{1}{t}{0}{1}  \ /  \ t \in \R \ \} \quad \mbox{and} \quad
B = \{ \ \matrice{\lambda}{t}{0}{\lambda^{-1}}  \ /  \ t \in \R , \lambda \in \R^+_\ast \ \}
$$
of the group $\PSL = \SL / \{\pm Id\}$.
We also consider a connected Lie group $G$ and the natural right actions of $U$ and $B$ on the product $\PSL \times G$ where every element of $\PSL$ acts trivially on the second factor $G$. We discuss the minimality of the right actions of $U$ and $B$ induced on the left quotient 
$X = \Gamma \rs \PSL \times G$ by a cocompact discrete subgroup of $\PSL \times G$. Recall that an action is said to be {\em minimal} if all the orbits are dense. 
\medskip 

In the case where $G$ is trivial, 
 assuming $\Gamma$ is torsion-free, the quotient $X = \Gamma \rs \PSL$ becomes the unit tangent bundle $T^1S$ to the compact 
hyperbolic surface $S = \Gamma \rs \Hy$ obtained from the Poincar\'e half-plane $\Hy$.
In 1936, G. A. Hedlund \cite{H} proved that the horocycle flow on $X$ is minimal (for an elementary proof, see \cite{G2}). In our context, this theorem can be reformulated as follows: 

\begin{namedtheorem}[Hedlund] Let $\Gamma$ be a discrete torsion-free cocompact subgroup of $\PSL$. Then the right a on $X = \Gamma \rs \PSL$ is minimal. 
\end{namedtheorem} 

\noindent
On the contrary, if $X$ is not compact,  M. Kulikov \cite{Ku} constructed an infinitely generated Fuchsian group without 
non-empty $U$-minimal sets. In the case of non uniform lattices of $\PSL$, like the modular subgroup $PSL(2,\mathbb{Z})$, the $U$-orbits are dense or periodic. Actually, it is known from \cite{D} that the $U$-action on $X$  is minimal if and only if $X$ is compact. 
\medskip 

When $G$ is not trivial, even assuming $X$ is compact, the $U$-action may be non minimal. This is the case for example when $G = \PSL$ and $\Gamma$ is the product of two cocompact Fuchsian groups. However, in this setting, we prove the following criterion: 

\begin{itheorem} \label{thm1}
Let $G$ be a connected Lie group and $\Gamma$ be a discrete cocompact subgroup of $\PSL \times G$. Then the right $U$-action on 
$X = \Gamma \rs \PSL \times G$ is minimal if and only if the corresponding $\PSL$-action is minimal.
\end{itheorem}

\noindent
Our proof of Theorem~\ref{thm1} does not use Ratner's famous Orbit-Closure Theorem \cite{R},  see also \cite{G2} and \cite{KSS} for an overview. 
In fact, some ideas will be applied in a non-homogeneous context.
\medskip

In the second part of this paper, we adopt a foliation point of view, which is natural in the previous context. For any connected Lie group $G$, 
the horizontal foliation of  $\PSL\times G$ by the fibres of the projection on the second factor $G$ is invariant by the action of $\Gamma$ and so 
induces a foliation on $X = \Gamma \rs \PSL\times G$ whose leaves are the orbits of the right $\PSL$-action. In fact, this action gives rise to a {\em $G$-Lie foliation} as defined in \cite{G} and \cite{Mo2}. As stated in a theorem by E. F\'edida \cite{F}, such a foliation is characterised as follows. Given a discrete group $\Gamma$ acting freely and properly discontinuously on a smooth manifold $\widetilde{M}$, a group homomorphism $h : \Gamma \to G$ and a locally trivial smooth fibration $\rho : \widetilde{M} \to G$ with connected fibres that is $\Gamma$-equivariant (i.e. $\rho(\gamma x)=h(\gamma)\rho(x)$ for all $\gamma \in \Gamma$ and for all $x \in \widetilde{M}$), the foliation $\widetilde{\F}$ by the fibres of $\rho$ induces a foliation $\F$ of $M = \Gamma \rs \widetilde{M}$, called {\em $G$-Lie foliation}, whose leaves are quotients of the fibres of $\rho$ by the kernel of $h$. 
\medskip 

Assume $\widetilde{M}$ is a connected Lie group $H$ equipped with a surjective morphism $\rho : H \to G$ and $\Gamma$ is a cocompact discrete subgroup of $H$. Like before, we obtain a $G$-Lie foliation on the homogeneous manifold $M = \Gamma \rs H$ whose leaves are the orbits of the right action of the kernel $K$ of $\rho$. Namely, they are diffeomorphic to $K \cap \Gamma \rs K$. Given a compact subgroup $K_0$ of $K$, we can modify this construction by considering $\widetilde{M} = H/K_0$ and $M = \Gamma \rs H / K_0$. According to \cite{G1}, any $G$-Lie foliation constructed by this method is called {\em homogeneous}. 
\medskip 

Let $\F$ a $G$-Lie foliation on a compact manifold $M$. When the leaves of $\F$ are equipped with a complete Riemann metric induced by a Riemann metric on $M$, we can define the {\em unit tangent bundle} $X = T^1\F$ of $\F$ as the vector bundle whose fibre $T^1_x \F$ at $x \in M$ is the unit tangent space  $T^1_x L_x$ to the leaf $L_x$ passing through $x$.
We say $\F$ is a $G$-Lie foliation by {\em hyperbolic surfaces} if the leaves of $\F$ are two-dimensional and the manifold $M$ is endowed with a complete Riemannian metric 
 whose restriction to each leaf 
has hyperbolic conformal type. 
Actually, according to 
\cite{C} and \cite{V}, we can assume (up to multiplication by a continuous function) that 
each restriction has constant negative curvature, namely each leaf is a hyperbolic surface. Once each leaf $L$ has a hyperbolic structure, its unit tangent bundle $T^1 L$ becomes diffeomorphic to the quotient of $\PSL$ by a discrete torsion-free subgroup. The transitive smooth right $\PSL$-action on $T^1 L$ extends to a leafwise smooth continuous right $\PSL$-action on $X = T^1\F$.  Notice that the
unit tangent bundles of the leaves of a $G$-Lie foliation $\F$ are always the leaves of a $G$-Lie foliation of $X = T^1\F$ which has the same transverse structure than $\F$. If $\F$ is a foliation by hyperbolic surfaces, this foliation is given by the continuous right $\PSL$-action described above. 
\medskip

In the situation described in Theorem~\ref{thm1}, the homogeneous manifold $X = \Gamma \rs \PSL \times G$ is the unit tangent bundle of the 
homogeneous $G$-Lie foliation $\F$ by hyperbolic surfaces on $M = \Gamma \rs \PSL/PSO(2,\R) \times G$ which is obtained where $H = \PSL \times G$, $K = \PSL$ is the kernel of the second projection $p_2 : \PSL \times G \to G$ and $K_0 = PSO(2,\R)$ is the compact stabiliser of $z=i$ for the $\PSL$-action on $\Hy$. 
Thus, from Theorem 1, we can derive the following generalisation of Hedlund's Theorem in the spirit of the work of M. Mart\'{\i}nez and  A. Verjovsky \cite{MV} on which we comment below: 

\begin{itheorem} \label{thm2}
 Let $X = T^1\F$ be the unit tangent bundle of a homogeneous $G$-Lie foliation $\F$ by hyperbolic surfaces of a compact manifold. If $\mathcal {F}$ is minimal, then the right $U$-action on $X$ is minimal. 
\end{itheorem}

\noindent
Note, however, that there are $G$-Lie foliations  which are not homogeneous \cite{HMM}.
A natural question arises when we replace $G$  with the quotient $G/G_0$ by a closed Lie subgroup $G_0$: does Theorem~\ref{thm2} remains valid for these more general foliations? They are {\em transversely homogeneous foliations} \cite{Bu} whose structure can be described in a similar way to that  of the $G$-Lie foliations. If $G_0$ is compact, we can construct by averaging a left-invariant Riemannian metric on $G$ that is also invariant by the right action of $G_0$. Then the distance between two $\PSL$-orbits in $X = \Gamma \rs \PSL \times G/G_0$ remains locally constant and therefore the right $\PSL$-action on $X$ defines a {\em Riemannian foliation} according to \cite{Mo1} and \cite{Mo2}. As in the Lie case, the homogeneous manifold 
$X = \Gamma \rs \PSL \times G/G_0$ is the unit tangent bundle of a homogeneous Riemannian foliation $\F$ constructed on the compact manifold
$M = \Gamma \rs \PSL/PSO(2,\R) \times G/G_0$. 
Using Molino's theory \cite{Mo1}, we extend Theorem~\ref{thm2} to this context in Corollary~\ref{corthm2}.
\medskip

In the third part of this paper, we show that, on the contrary, Theorem~\ref{thm2} does not hold for general transversely homogeneous foliations where $G_0$ is a non-compact closed Lie subgroup of $G$: 

\begin{itheorem} \label{thm3}
There are minimal transversely homogeneous foliations on compact manifolds such that neither the $U$-action nor the $B$-action on its unit tangent bundles are minimal. Moreover, there is such an example admitting a unique $B$-minimal set which is not $U$-minimal. 
\end{itheorem} 

As mentioned before, the problem of generalising Hedlund's Theorem for compact foliated manifolds by hyperbolic surfaces has been discussed by M. Mart\'{\i}nez and  A. Verjovsky in several versions of their article \cite{MV}. Theorems~\ref{thm1}~and~\ref{thm2} give an affirmative answer to the initial conjecture for homogeneous Lie foliations, also valid for homogeneous Riemannian foliations, while Theorem~\ref{thm3} gives a negative answer in the transversely homogeneous setting. Nevertheless, according to the second version of \cite{MV}, the question can be reformulated in the following way:

\begin{namedquestion}[Mart\'{\i}nez-Verjovsky]
Let $X = T^1\F$ be the unit tangent bundle of a compact foliated manifold whose leaves are hyperbolic surfaces. Is it true that the right $U$-action on $X$ is minimal if and only if the right $B$-action is minimal?
\end{namedquestion}

\noindent
We complete the paper with some comments on this question.

\subsection*{Acknowledgements} We thank Bertrand Deroin and Ga\"el Meigniez for their accurate comments.  This work has been partially supported by the Ministry of Science and Innovation - Government of Spain (Grant MTM2010-15471) and IEMath Network CN 2012/077.

\section{Proof of Theorem~\ref{thm1}} \label{Section1}

Let $G$ be a connected Lie group. Let $\Gamma$ be a discrete subgroup of the Lie group $H = \PSL \times G$ acting on $H$ by left  translation.  We denote by $p_1$ and $p_2$ the first and second projection  of $H = \PSL \times G$ onto $\PSL$ and $G$ respectively. Any subgroup $F$ of $\PSL$ acts on the quotient $X = \Gamma \rs \PSL \times G$ by right translation
$$
 \Gamma (f,g)f' = \Gamma (ff', g) 
$$
for all $(f,g) \in H$ and for all $f' \in F$. In the following, we will replace $F$ with $U$, $B$ or $\PSL$. By duality, the right $F$-action on $X$ is minimal if and only if the action of $\Gamma$ on the quotient $\PSL/F \times G$ by left translation is minimal. In particular, the right $\PSL$-action on $X$ is minimal if and only if $\overline{p_2(\Gamma)} = G$. If the right $F$-action is minimal, then $p_1(\Gamma)$ acts minimally on 
$\PSL/F$ and $\overline{p_2(\Gamma)} = G$. For $F = B$, we prove: 

\begin{proposition} \label{charactB}  Let $G$ be a connected Lie group and let $\Gamma$ be a discrete subgroup of $H = \PSL \times G$. Then the right $B$-action on $X$ is minimal if and only if the following two properties hold: 
\begin{list}{\labelitemi}{\leftmargin=5pt}

\item[(i)]  $p_1(\Gamma)$ acts minimally on 
$\PSL/B$, 

\item[(ii)] $\overline{p_2(\Gamma)} = G$, or equivalently the right $\PSL$-action on $X$ is minimal. 
\end{list}
\end{proposition}

\noindent
The proof of Proposition~\ref{charactB} uses hyperbolic geometry.  Let $\Hy = \{ z \in \C / Im \, z > 0 \}$ be the Poincar\'e half-plane equipped with the hyperbolic distance $d$. The action of $\PSL$ on $\Hy$ by isometries extends to a $\PSL$-action on its boundary 
$\partial \Hy = \R \cup \{\infty\}$. This action is conjugated to the right $\PSL$-action on $\PSL/B$. Since $\SL$ acts transitively on 
$\R^2 - \{0\}$ and $U$ is the stabiliser of the vector $e_1 = (1,0)$, the homogeneous manifolds $\PSL/U$ and  $\PSL/B$ are diffeomorphic to the linear space $E = \R^2 - \{0\}/\{\pm Id\}$ and the projective line $\mathbb{R}P^1$ respectively. Before we prove Proposition~\ref{charactB}, we state the following key lemma: 

\begin{lemma} \label{keylemma}
Let $\{ f_n\} _{n \geq 0}$ be a sequence of elements of $\PSL$. If for some $z \in \Hy$, there are points 
$\xi^+$ and $\xi^-$ in $\partial \Hy$ such that
$$
\lim_{n \to +\infty} f_n(z) = \xi^+  \quad \mbox{and} \quad \lim_{n \to +\infty} f_n^{-1}(z) = \xi^-,
$$
then for every point $\xi \neq \xi^-$ in $\Hy \cup \partial \Hy$, we have:
$$
\lim_{n \to +\infty} f_n(\xi) = \xi^+.
$$
\end{lemma}

\begin{proof}
For each point $\xi \in \Hy$, we have 
$\lim_{n \to +\infty} f_n(\xi) = \xi^+$ since $d(f_n(\xi),f_n(z)) = d(\xi,z)$. For $\xi \neq \xi^-$ in $\partial \Hy$, we choose $\xi'  \in \partial \Hy$  different from $\xi$ and $\xi^-$and 
a geodesic $\alpha : \R \to \Hy$ joining $\xi$ to $\xi'$, that is, $\xi= \lim_{t \to -\infty} \alpha(t)$ and $\xi' = \lim_{t \to +\infty} \alpha(t)$. 
If we denote by $\alpha_n = f_n \scirc \alpha$ the geodesic joining 
$f_n(\xi)$ to $f_n(\xi')$, then $d(f_n^{-1}(z), \alpha(t)) = d(z,\alpha_n(t))$  for all $t \in \R$. Since $\lim_{n \to +\infty} f_n^{-1}(z) = \xi^-$ and $\xi^-$ is different from $\xi$ and $\xi'$, we have
$\lim_{n \to +\infty} d(z,\alpha_n(t)) = + \infty$
for all $t \in \R$. It follows that the sequence of geodesics $\alpha_n$ converges to a point $\zeta \in \partial \Hy$. This implies that 
$\lim_{n \to +\infty} f_n(\alpha(t)) = \zeta$
for all $t \in R$. Now, since $\alpha(t)$ belongs to $\Hy$, we have
$\lim_{n \to +\infty} f_n(\alpha(t)) = \xi^+$ and hence $\lim_{n \to +\infty} f_n(\xi) = \zeta = \xi^+$.
\end{proof}

\begin{proof}[Proof of Proposition~\ref{charactB}]
By duality, it is enough to prove the action of $\Gamma$ on $\partial \Hy \times G$ is minimal when $p_1(\Gamma)$ acts minimally on $\partial \Hy$ and $\overline{p_2(\Gamma)} = G$. This second condition allows us to choose a non stationary sequence $\{ g_n\} _{n \geq 0}$ in  $p_2(\Gamma)$ that converges to the identity element $1$ of $G$. Then there is a sequence $\{ f_n\} _{n \geq 0}$ in $\PSL$ such that $\gamma_n = (f_n,g_n) \in \Gamma$ for all $n \geq 0$. Since $\Gamma$ is discrete, this sequence $\{ f_n\} _{n \geq 0}$ is not bounded. Thus, without loss of generality, we can assume that the sequences 
$\{ f_n(z)\} _{n \geq 0}$ and $\{ f_n^{-1}(z)\} _{n \geq 0}$ converge to some points $\xi^+$ and $\xi^-$ in $\partial \Hy$ for some $z \in \Hy$. 
For each point $\xi \neq \xi^-$ in $\partial \Hy$, we deduce from the key lemma~\ref{keylemma} that 
$$
(\xi^+,g) = \lim_{n \to +\infty} (f_n(\xi),g_ng) =  \lim_{n \to +\infty} \gamma_n(\xi,g) \in \overline{\Gamma(\xi,g)}
$$
for all $g \in G$. More generally, assuming that  $\xi \neq f(\xi^-)$  for some  $f \in p_1(\Gamma)$ and replacing 
$\gamma_n$ with $\gamma'\gamma_n(\gamma')^{-1}$ where $\gamma' = (f,g') \in \Gamma$, we have: 
$$
(f(\xi^+),g)  = \lim_{n \to +\infty} (ff_n(f^{-1}(\xi)),g'g_n(g')^{-1}g)   = \lim_{n \to +\infty} \gamma'\gamma_n(\gamma')^{-1}(\xi,g)  \in \overline{\Gamma(\xi,g)}.
$$
Thus, if $\xi \in \partial \Hy$ does not belong to the orbit $p _1(\Gamma)\xi^-$, then $\overline{p_1(\Gamma)\xi^+}\times \{g\} \subset \overline{\Gamma(\xi,g)}$. Using the minimality of the action of $p_1(\Gamma)$ on $\partial  \Hy$, we get $\partial \Hy \times \{g\} \subset \overline{\Gamma(\xi,g)}$ for all $g \in G$. Now, since $\overline{p_2(\Gamma)} = G$, it follows that $\overline{\Gamma(\xi,g)} = \partial \Hy \times G$.
Finally, assume that $\xi = f(\xi^-)$ for some $f \in p_1(\Gamma)$. 
Since the $p_1(\Gamma)$ acts minimally on $\partial  \Hy$ and contains unbounded sequences like 
$\{ f_n\} _{n \geq 0}$, either $p_1(\Gamma)$ is dense in $\PSL$ or $p_1(\Gamma)$ is a Fuchsian group of first kind (i.e having $\partial \Hy$ as limit set). This implies that there exists $\gamma'= (f',g') \in \Gamma$ such that the sequence $(f')^k(\xi^+)$ converges to a point $\xi' \notin p_1(\Gamma)\xi^-$ when $k$ goes to $+\infty$ 
and $(f')^k(\xi^-) \neq \xi$ for all $k \geq 0$. 
So the sequence 
$(\gamma')^k =  ((f')^k,(g')^k ) \in \Gamma$ verifies:
$$
\lim_{n \to +\infty} (\gamma')^k\gamma_n(\gamma')^{-k}(\xi,g) = ((f')^k(\xi^+), g) 
$$
and therefore $(\xi',g)$ belong to $\overline{\Gamma(\xi,g)}$. Since $\xi' \notin p_1(\Gamma)\xi^-$, according to the previous step, $\Gamma(\xi',g)$ is dense in $\partial \Hy \times G$ and hence $\Gamma(\xi,g)$ is also dense.
\end{proof} 

Theorem~\ref{thm1} really concerns cocompact discrete subgroups. Before we deal with this case, let us introduce the notion of semi-parabolic element of the Lie group $\PSL\times G$. Thus, we say that $(f,g) \in \PSL\times G$ is {\em semi-parabolic} if $f$ is conjugated in $\PSL$ to an element $u \neq Id$ in $U$. The existence of semi-parabolic elements in $\Gamma$ is related to the behaviour of the right $D$-action on $X$ where 
 $$
D = \{ \ \matrice{\lambda}{0}{0}{\lambda^{-1}}  \ /  \lambda > 0 \ \} \quad \mbox{and} \quad D^+ = \{ \ \matrice{\lambda}{0}{0}{\lambda^{-1}}  \ /  \lambda > 1 \ \}
$$ 
are the diagonal group and its strictly positive cone.

\begin{lemma} \label{splemma}
If $\Gamma$ contains a semi-parabolic element, then there are divergent positive semi-orbits with respect the right $D^+$-action on $X$.
\end{lemma}

\begin{proof} Assume that $\Gamma$ contains a semi-parabolic element 
$\gamma = (fuf^{-1},g)$ where $u \in U - \{Id\}$, $f \in \PSL$ and $g \in G$. Given $g' \in G$,  we set $x = \Gamma(f,g') \in X$ and we prove that 
$xD^+$ diverges. Suppose on the contrary that the sequence $\{xd_n\}_{n \geq 0}$ converges for some non-bounded sequence
 $\{d_n\}_{n \geq 0}$ in $D^+$. Put 
$$
d_n = \matrice{\lambda_n}{0}{0}{\lambda_n^{-1}}
$$
such that $\lambda _n \to +\infty$. Also write 
$$
u = \matrice{1}{t}{0}{1}
$$
with $t \neq 0$.  By hypothesis, there exists a sequence $\{\gamma_n\}_{n\geq 0}$ in $\Gamma$ such that 
$\gamma_n(f,g')d_n$ converge to some element $(f'',g'')$ in $H$. Notice that 
$$
\gamma_n(f,g')d_n = \gamma_n\gamma^{-1}\gamma(f,g')d_n = \gamma_n\gamma^{-1}(fud_n,gg') = 
\gamma_n\gamma^{-1}(fd_nd_n^{-1}ud_n,gg') 
$$
and 
$$
\lim_{n \to +\infty} d_n^{-1}ud_n = \lim_{n \to +\infty} \matrice{1}{t\lambda_n^{-2}}{0}{1} = Id.
$$
We deduce that the sequence $\gamma_n\gamma^{-1}(fd_n,gg')$ also converges to $(f'',g'')$. Now, since
$$
\gamma_n\gamma^{-1}(fd_n,gg') = \gamma_n\gamma^{-1}\gamma_n^{-1} \Big( \gamma_n(f,g')d_n\Big) (Id,(g')^{-1}g g')
$$
and 
$$
\lim_{n \to +\infty} \gamma_n(f,g')d_n 
= (f'',g''), 
$$
it follows that $\gamma_n\gamma^{-1}\gamma_n^{-1}$ converges to $(Id,g''(g')^{-1}g^{-1} g'(g'')^{-1})$ in $H$. Since $\Gamma$ is discrete, for $n$ large enough, we have $p_1(\gamma_n\gamma^{-1}\gamma_n^{-1}) = Id$ and therefore $u = Id$ 
contradicting the hypothesis. 
\end{proof}

Let us assume $X$ is compact. From Lemma~\ref{splemma}, we have immediately:

\begin{proposition} \label{semiparabolic}
 If  $X = \Gamma \rs \PSL \times G$ is compact, then $\Gamma$ does not contain semi-parabolic elements. \qed
\end{proposition}

\noindent
Before we reformulate Proposition~\ref{charactB} in the cocompact case, let us recall the following classification lemma: 

\begin{speciallemma} \label{lemmaD}
Let $\Delta$ be a subgroup of $\PSL$ and denote by $\oclosure{\Delta}$ the connected component of the identity of its closure $\overline{\Delta}$. If $\Delta$ is neither discrete nor dense, then $\oclosure{\Delta}$ is conjugated to $PSO(2,\R)$ or a Lie subgroup of $B$. \qed 
\end{speciallemma}

\begin{proposition} \label{cocompactB}
Let $\Gamma$ be a cocompact discrete subgroup of $H = \PSL \times G$. Denote by $X$ the compact quotient by left translation. Then the right $B$-action on $X$ is minimal if and only if $\overline{p_2(\Gamma)} = G$, or equivalently if 
the right $\PSL$-action on $X$ is minimal.
\end{proposition}

\begin{proof} We have only to prove the \lq if\rq~part. Now, according to Proposition~\ref{charactB}, it is enough to show that 
$\Delta = p_1(\Gamma)$ of $\PSL$ acts minimally on  $\partial \Hy$. We distinguish two cases, depending on whether this group is discrete or not. 
\medskip 

If $\Delta = p_1(\Gamma)$ is discrete, the surface $\Delta\rs \Hy$ is compact because  $X$ is compact too.  
It follows that any orbit of $p_1(\Gamma)$ in $\Hy$ accumulates to its full boundary $\partial \Hy$. 
In other words, the limit set of $\Delta$ is equal to $\partial \Hy$.
Since the limit set of any non-elementary Fuchsian group is minimal, we deduce that the action of $\Delta$ on $\partial \Hy$ is minimal. 
\medskip 

If $\Delta = p_1(\Gamma)$ is non-discrete but dense, then the action of $\Delta$ on $\partial \Hy$ is still minimal. Otherwise, according to 
Classification Lemma~\ref{lemmaD} and using the fact that $\Delta$ normalises $\oclosure{\Delta}$, we deduce that $\Delta$ is conjugated to a subgroup 
of $PSO(2,\R)$ or $B$. Since $X$ is compact, the first case is excluded. Assuming $\Delta \subset fBf^{-1}$ for some $f \in \PSL$, we have
$[\Delta,\Delta ] \subset fUf^{-1}$ and therefore $\Delta$ is abelian as a consequence of Proposition~\ref{semiparabolic}. It follows that 
$\Delta$ is conjugated to a subgroup of $D$, which contradicts the compactness of $X$. 
\end{proof}

\begin{proof}[Proof of Theorem~\ref{thm1}] According to Proposition~\ref{cocompactB}, it is enough to prove that if the right $B$-action is minimal, then the right $U$-action is minimal too. By compactness of $X$, the $U$-action has a non-empty minimal set $\M$. Let us prove $\M$ is $B$-invariant so that $\M = X$ and the right $U$-action minimal. 
\medskip 

Let $h = (f,g)$ be an element of $H = \PSL \times G$ such that $x = \Gamma h \in \M$. 
Since $\overline{xU} = \M$, there are elements $\gamma_n = (\gamma_{1n},\gamma_{2n}) \in \Gamma$ and 
$$u_n  = \matrice{1}{t_n}{0}{1} \in U
$$ 
with $t_n \to +\infty$ such that 
$$
 \lim_{n \to +\infty} \gamma_n (f,g)u_n =  \lim_{n \to +\infty} (\gamma_{1n}fu_n,\gamma_{2n}g) = (f,g) = h
$$
If we write
$f_n =  f^{-1}\gamma_{1n}fu_n$, $g_n =  g^{-1}\gamma_{2n}g$, and $h_n = (f_n,g_n)$, 
the sequence $$hh_n = (ff_n,gg_n) = \gamma_nhu_n$$ converges to $h$ so that
$$\lim_{n \to +\infty}  h_n = \lim_{n \to +\infty}  (f_n,g_n) = (Id,e).$$
Notice that the sequence $\{\gamma_{1n}\}_{n \geq 0}$ does not admit any convergent subsequence because $t_n \to +\infty$. 
On the other hand, 
since $hh_n = \gamma_nhu_n$ represents  the class $xu_n$ in the orbit $xU$, the element $h_n = (f_n,g_n) \in H$ belongs to the set 
\begin{eqnarray*}
H_\M  = \{ h' \in H / \M h' \cap \M \neq \emptyset \}
\end{eqnarray*}
having the following properties:

\begin{lemma} \label{invset}
The set $H_\M$ is a closed subset of $H = \PSL \times G$ which 
is invariant under the right and left $U$-actions on $H$.
\end{lemma} 

\begin{proof} Let  $h'_n \in H_\M$ be a sequence that converges to some element $h' \in H$. By definition, for any $n \in \N$, there is $x_n = 
\Gamma h_n \in \M$ such that $x_n h'_n = \Gamma h_nh'_n  \in  \M$. By compactness of $\M$ and replacing the sequence with some subsequence if necessary, we may assume that the sequence $x_n$  converges to a class $x=\Gamma h$ in $\M$. Then $x h' = 
\lim_{n \to +\infty} x_nh'_n \in \M$ and hence $h' \in H_\M$. 
\medskip 

Let us prove $H_\M$ is invariant under the right and left $U$-actions on $\PSL \times G$.  Indeed, since $\M u^{-1} = \M$, we have:
$$
\M uh \cap \M =  \M u^{-1}uh \cap \M = \M h \cap \M \neq \emptyset
$$
for all $u \in U$ and for all $h \in H_\M$. Likewise, we have:
$$
\M hu \cap \M = (\M h \cap \M u^{-1})u = (\M h \cap \M )u \neq \emptyset
$$
proving the right invariance.
\end{proof}

Returning to the proof of Theorem~\ref{thm1}, we have: 

\begin{lemma} \label{hn}
There exists $k \in \N$ such that $f_n \notin B$ for $n \geq k$.
\end{lemma}

\begin{proof} Let us assume on the contrary that for every $k \in \N$, there exists $n_k \geq k$ such that $f_{n_k} \in B$.  Then $f^{-1}\gamma_{1n_k}f = f_{n_k}u_{n_k}^{-1} \in B$ and hence $\gamma_{1n_k} \in fBf^{-1}$. It follows that $[\gamma_{1n_k},\gamma_{1n_{k'}}] \in \Gamma \cap fUf^{-1}$ for all $k,k' \geq 0$. But according to Proposition~\ref{semiparabolic}, $\Gamma$ does not contain semi-parabolic elements and therefore $[\gamma_{1n_k},\gamma_{1n_{k'}}] = Id$. Then there exists $u \in U$ such that 
$$
f^{-1}\gamma_{1n_k}f = u\matrice{\lambda_{n_k}}{0}{0}{\lambda_{n_k}^{-1}}u^{-1}
$$ 
for all $k \geq 0$. Since the sequence $\{\gamma_{1n_k}\}_{k \geq 0}$ does not converge, the sequence $\{\lambda_{n_k}\}_{k \geq 0}$ is not bounded, which is impossible because the matrices 
$f_{n_k} = f^{-1}\gamma_{1n_k}f u_{n_k}$ converge to $Id$ and hence the vectors
$f_{n_k} e_1 = (\lambda_{n_k}, 0)$ converge to $e_1 = (1,0)$. 
\end{proof}

To conclude, let us put 
$$
f_n = \matrice{a_n}{b_n}{c_n}{d_n}
$$
where $c_n \neq 0$ according to Lemma~\ref{hn}. For every $\alpha \in \R^\ast_+$, take 
$$
u'_n = \matrice{1}{\frac{\alpha - a_n} {c_n}}{0}{1} \quad \mbox{ and }   \quad 
u''_n = \matrice{1}{-\frac{1}{\alpha}(b_n+d_n \frac{\alpha - a_n}{c_n})}{0}{1}
$$
in $U$. From Lemma~\ref{invset}, as $h_n = (f_n,g_n) \in H_\M$, we have: 
$$
u'_nh_nu''_n = (u'_nf_nu''_n,g_n) = (\matrice{\alpha}{0}{c_n}{\alpha^{-1}},g_n)  \in H_\M
$$
Since $\lim_{n \to +\infty} c_n = 0$ and $\lim_{n \to +\infty} g_n = e$ , we deduce that
$$
(\matrice{\alpha}{0}{0}{\alpha^{-1}},e) \in H_\M.
$$
This means that 
$$
\M_\alpha = \M \matrice{\alpha}{0}{0}{\alpha^{-1}} \cap \M \neq \emptyset
$$
for all $\alpha \in \R^\ast$. Since 
$$
\matrice{\alpha}{0}{0}{\alpha^{-1}} \matrice{1}{t}{0}{1} = \matrice{1}{\alpha^2 t}{0}{1} \matrice{\alpha}{0}{0}{\alpha^{-1}},
$$
the set $\M _\alpha$ is a $U$-invariant closed subset of $\M$. By minimality, we have $\M_\alpha = \M$ and therefore $\M$ is $D$-invariant, i.e. 
$$
\M \matrice{\alpha}{0}{0}{\alpha^{-1}} = \M 
$$
for all $\alpha \in \R^\ast_+$. So $\M$ is also $B$-invariant and hence $\M = X$ from Proposition~\ref{cocompactB}.
\end{proof}

In the particular case where $G$ is trivial, we have just given a simple proof of Hedlund's Theorem, which is essentially the one that Ghys gave in \cite{G2}.
We illustrate the general situation with two examples: 

\begin{examples}  \label{examplesLie}
(i) 
According to Theorem C of \cite{B}, if $G = \PSL$, then $H = \PSL \times \PSL$ admits discrete uniform subgroups $\Gamma$. If $\Gamma$ is irreducible, then $p_1(\Gamma)$ and $p_2(\Gamma)$ are dense in $\PSL$. In particular, the natural right 
$\PSL$-action on $X = \Gamma \rs H$ is minimal. From Theorem~\ref{thm1}, the natural right $U$-action is minimal too. 
\medskip 

\noindent
(ii) In \cite{BGSS}, the authors proved that any torsion-free cocompact Fuchsian group $\Gamma$ can be realised as a dense subgroup of  $G = SO(3,\R)$. Let $h$ be an injective representation of $\Gamma$ into $SO(3,\R)$ and consider the free and properly discontinuous action of $\Gamma$ on $H = \PSL \times SO(3,\R)$ given by $\gamma.(f,g) = (\gamma f, h(\gamma)g)$ for all $\gamma \in \Gamma$ and for all $(f,g) \in H$. This allows us to see $\Gamma$ as a cocompact discrete subgroup of $H$. Since $h(\Gamma)$ is dense in $SO(3,\R)$, by applying Theorem~\ref{thm1}, we conclude that the natural right $U$-action on $X =  \Gamma \rs H$ is minimal. 
\end{examples}

\section{Proof of Theorem~\ref{thm2}} \label{Section2}

Let $G$ be a connected Lie group and let $\mathfrak{g}$ be its Lie algebra. Right $\PSL$-actions on homogeneous manifolds $X = \Gamma \rs \PSL \times G$ are examples of smooth $G$-Lie foliations.  This type of foliations has been classically defined using smooth foliated cocycles with values in $G$ or smooth differential $1$-forms with values in $\mathfrak{g}$, see \cite{G}, \cite{Mo1} and \cite{Mo2}. However, in our context, it is more convenient to use the following criterion as definition: 

\begin{theorem}[\cite{F}] \label{Fedida}
A smooth foliation $\F$ on a compact connected manifold $M$ is a $G$-Lie foliation if and only if  there are 
\begin{list}{\labelitemi}{\leftmargin=5pt}

\item[(i)] a discrete group $\Gamma$ acting freely and properly discontinuously on a manifold $\widetilde{M}$,

\item[(ii)] a group homomorphism $h : \Gamma \to G$,

\item[(iii)] a $\Gamma$-equivariant locally trivial smooth fibration $\rho : \widetilde{M} \to G$ with connected fibres,  
 \end{list} 
such that $M = \Gamma \rs \widetilde{M}$ and $\F$ is induced by the foliation $\widetilde{\F}$ of $\widetilde{M}$ whose leaves are the fibres of $\rho$.  The group $\Gamma$ is called the {\em holonomy group}  of $\F$. 
\end{theorem}

\noindent
Assume that the leaves of $\F$ are $2$-dimensional. Given a complete Riemannian metric $g_0$ on $M$, $\F$ is said to be a foliation {\em by hyperbolic surfaces} if the restriction of $g_0$ to each leaf has hyperbolic conformal type. Actually, according to the Uniformisation Theorem of \cite{C} and \cite{V} which remains valid for any foliation by hyperbolic surfaces, there exists a (leafwise smooth) continuous function $u : M \to \R$ such that the restriction of the conformal Riemann metric $g = ug_0$ to each leaf has constant negative curvature equal to $-1$. Then each leaf $L$ is the quotient of
the Poincar\'e half-plane $\Hy$ by the action of a discrete torsion-free subgroup $\Gamma_L$ of $\PSL$. Since $\PSL$ acts freely and transitively on  $T^1 \Hy$, the unit tangent bundle $T^1 L$ is diffeomorphic to $\Gamma_L \rs PSL$. The natural smooth right $PSL$-action 
on $T^1 L \cong \Gamma_L \rs PSL$ extends to continuous global  $PSL$-action on $T^1\F$. The {\em  foliated horocycle and geodesic  flows} on $T^1\F$ are defined by the corresponding $U$-action and $D$-action, which coincide with the usual geodesic and horocycle flow on $T^1 L$ in restriction to each leaf $L$. 
In the case of the $G$-Lie foliations, we have also the following additional property: 

\begin{proposition} \label{principal}
Let $\F$ be a $G$-Lie foliation by hyperbolic surfaces of a compact connected manifold $M$. Then the developing map $\rho$ is trivial, so
$\widetilde{M}$ is homeomorphic to a product $L \times G$. Moreover, the homeomorphism becomes a diffeomorphism if and only if $\F$ admits a smooth uniformisation. 
\end{proposition} 

\begin{proof} Firstly, by replacing $\widetilde{M}$ and $G$ with the universal coverings of $M$ and $G$, we can assume that $\widetilde{M}$ and $G$ are simply connected. Furthermore, since the second homotopy group of the Lie group $G$ is trivial \cite{Ca}, we can use the homotopy sequence of $\rho$ to deduce that the fibre $L$ is also simply connected and hence $L = \Hy$. Then the natural right $\PSL$-action on $X = T^1\F$ lifts to a free and proper $\PSL$-action on $\widetilde{X} =T^1 \widetilde{\F}$ whose orbits are diffeomorphic to the unit tangent bundles of the fibres of $\rho$. It follows that $\widetilde{X}$ is a continuous principal $\PSL$-bundle over $G$, which becomes smooth if and only if $\F$ has a smooth uniformisation. 
By construction, the bundle map $\widetilde{\rho}  :  \widetilde{X} \to G$ is the developing map of the $G$-Lie foliation on $X$ whose leaves are the 
$\PSL$-orbits.  
\medskip

On the other hand, since the structure group $\PSL$ retracts by deformation on the stabiliser $PSO(2,\R)$ of $z=i$ in $\Hy$, the $\PSL$-bundle $\widetilde{X}$ admits a reduction to $PSO(2,\R)$. This means that there exists a continuous principal $PSO(2,\R)$-bundle $P$ over $G$ such that 
$\widetilde{X}$ is isomorphic to the continuous principal $\PSL$-bundle associated to $P$, that is, $\widetilde{X}$ is homeomorphic to the quotient of $P \times \PSL$ by the diagonal $PSO(2,\R)$-action that is given by 
$(p,f)r = (pr,r^{-1}f)$ for all $(p,f) \in P \times  \PSL$ and all $r \in PSO(2,\R)$. 

\medskip
Finally, let us recall that principal $PSO(2,\R)$-bundles over $G$ are classified by the Euler class in the integer cohomology group $H^2(G,\Z)$, see for example \cite{BT}.
Actually, according to the universal coefficient theorem (see also \cite{BT}), this group
$H^2(G,\Z) = Hom(H_2(G,\Z),\Z) \oplus Ext(H_1(G,\Z),\Z)$
is trivial because the homotopy groups $\pi_1(G)$ and $\pi_2(G)$ are trivial. 
Briefly, the principal $PSO(2,\R)$-bundle $P$ is trivial, so there is a homeomorphism 
$\varphi : P \to PSO(2,\R) \times G$ which is equivariant for the natural right $PSO(2,\R)$-actions.
By sending each $PSO(2,\R)$-orbit represented by $(p,f) \in P \times \PSL$ with $\phi(p) = (r,g)$ to the point $\Phi((p,f)PSO(2,\R))  = (rf,g)$
in $\PSL \times G$, we obtain a well-defined $\PSL$-equivariant homeomorphism
$\Phi : \widetilde{X} \to \PSL \times G$ 
such that $\widetilde{\rho} = p_2 \scirc \Phi$. Now, by passing to the quotient by the corresponding $PSO(2,\R)$-action, $\Phi$ induces a homeomorphism 
$\overline{\Phi} :  \widetilde{M} \to \Hy \times G$ such that $\rho = p_2 \scirc \overline{\Phi}$. From the previous discussion, it is also clear that 
$\overline{\Phi}$ is a diffeomorphism if and only if $\F$ admits a smooth 
uniformisation. 
\end{proof} 

Now, we restrict our attention to the notion of homogeneous $G$-Lie foliation as defined in the introduction and illustrated by Examples~\ref{examplesLie}. Recall that a $G$-Lie foliation $\F$ on a compact manifold $M$ is said to be {\em homogeneous} if there are a connected Lie group $H$ equipped with a surjective morphism $\rho : H \to G$, a compact subgroup $K_0$ of the kernel $K$ of $\rho$ and a cocompact discrete subgroup $\Gamma$ of $H$ such that $M$ is diffeomorphic to $\Gamma \rs H / K_0$ and 
$\F$ is conjugated to the foliation induced by the right $K$-action on $H$. In the case where $\F$ is a two-dimensional foliation by hyperbolic surfaces, we can assume $K = \PSL$ and $K_0 = PSO(2,\R)$. 

\begin{proposition} \label{homogeneous}Let  $\F$ be $G$-Lie foliation by hyperbolic surfaces of a compact connected manifold $M$. Then the following conditions are equivalent: 
\begin{list}{\labelitemi}{\leftmargin=5pt}

\item[(i)] The foliation $\F$ is homogeneous. 

\item[(ii)] The right $\PSL$-action on $X=T^1 \F$ is conjugated to the natural right $\PSL$-action on some quotient of the Lie group 
$H = \PSL \times G$ by a cocompact discrete subgroup. 

\item[(iii)] Up to conjugation by a diffeomorphism between $\widetilde{M}$ and $\Hy \times G$, the holonomy group $\Gamma$ acts diagonally on $\Hy \times G$, that is, 
$\gamma.(z,g) = (\gamma(z),\rho(\gamma)g)$ for all $\gamma \in \Gamma$ and for all 
 $(z,g) \in \Hy \times G$. 
\end{list}
\end{proposition} 

\begin{proof} As in the proof of Proposition~\ref{principal}, we are assuming that $\widetilde{M}$ and $G$ are simply connected. We also keep the notation just described.  Now we prove the proposition through the following cycle of implications: 
\medskip 

\noindent
$(i) \Rightarrow (ii)$ According to a result of H. Cartan, see for example \cite{S}, the Lie algebra $\mathfrak{h}$ of the Lie group $H$ split into the direct sum $\mathfrak{h} = \mathfrak{sl}(2,\R) \oplus \mathfrak{g}$ of the Lie algebras of $\PSL$ and $G$. 
Then the simply connected Lie group $\widetilde{H}$ integrating $\mathfrak{h}$ split into the product $\widetilde{PSL}(2,\R) \times G$ where $\widetilde{PSL}(2,\R)$ is the universal covering of $\PSL$.  Moreover, the fundamental group of $H$ is isomorphic to the fundamental group of $\PSL$. If follows that 
$H$ is isomorphic to $\PSL \times G$. 
\medskip 

\noindent
$(ii) \Rightarrow (iii)$ By hypothesis, the action of the holonomy group $\Gamma$ on $\widetilde{X}$ is conjugated to the action of some discrete cocompact subgroup of $H = \PSL \times G$. Then the $\Gamma$-action on $\widetilde{M}$ is conjugated to a diagonal action on  $H / PSO(2,\R)  \cong \Hy \times G$. 
\medskip 

\noindent
$(iii) \Rightarrow (i)$ Assume the holonomy group $\Gamma$ acts diagonally on the universal covering $\widetilde{M}$. Then, up to conjugation by a diffeomorphism between $\widetilde{M}$ and $\Hy \times G$, the $\Gamma$-action on $\widetilde{X} = T^1\widetilde{\F}$ is conjugated to the natural left action of a discrete cocompact subgroup of $H = \PSL \times G$ and therefore $X = T^1\F$ becomes diffeomorphic to the corresponding quotient of $H$, endowed with the natural $\PSL$-action. \end{proof} 

In \cite{HMM}, G. Hector, S. Matsumoto and G. Meigniez constructed an example of minimal $\PSL$-Lie foliation by hyperbolic surfaces which is not homogeneous. Comparing with Propositions~\ref{principal}~and~\ref{homogeneous}, the universal covering $\widetilde{M}$ is diffeomorphic to $\Hy \times \PSL$, but he holonomy group $\Gamma$  does not act diagonally. However, in the homogeneous setting, Theorem~\ref{thm2} can be immediately deduced as a corollary of Proposition~\ref{homogeneous} and Theorem~\ref{thm1}: 

\begin{proof}[Proof of Theorem~\ref{thm2}] Let $\F$ be a $G$-Lie foliation of a compact connected manifold $M$ whose leaves are hyperbolic surfaces. The natural right $\PSL$-action on the unit tangent bundle $X = T^1 \F$ is minimal if and only if $\F$ is minimal because they have the same holonomy representation $h : \Gamma \to G$. Assuming $\F$ is homogeneous and using Proposition~\ref{homogeneous}, we can apply Theorem~\ref{thm1} to deduce that $\PSL$-minimality and $U$ minimality are equivalent on $X = T^1 \F$.
\end{proof}Ê

As we already mentioned in the introduction, when we replace the Lie group $G$ with the quotient $G/G_0$ by a compact Lie subgroup $G_0$, we obtain an example of Riemannian foliation where the distance between two leaves (deduced from a left-invariant Riemannian metric on $G$ that is also invariant by the right $G_0$-action) remains locally constant. In general, a foliation $\F$ is said to be {\em Riemannian} when the distance between two leaves verifies this property, see  \cite{G}, \cite{Mo1} and \cite{Mo2}. 

\begin{examples} \label{examplesRiemann}
(i) Consider the $\PSL$-Lie foliation constructed in Examples \ref{examplesLie}.(i) by quotienting the Lie group $H = \PSL \times \PSL$ by an irreducible cocompact discrete subgroup $\Gamma$. Assuming $\Gamma$ torsion-free, the $\Gamma$-action on \mbox{$\Hy \times \Hy$}
$\cong H / PSO(2,\R) \times PSO(2,\R)$ is free and proper, so the horizontal foliation of $\Hy \times \Hy$ induces a minimal Riemannian foliation $\F$ on the quotient manifold $M = \Gamma \rs \Hy \times \Hy$. The foliation $\F$ lifts to a minimal $\PSL$-Lie foliation $\F_T$ on $E_T = \Gamma \rs \Hy \times \PSL$ defined by the representation of $\Gamma$ onto the dense subgroup $p_2(\Gamma)$ of $\PSL$. 
Notice that $E_T$ is a principal $PSO(2,\R)$-bundle on $M$ whose elements are positively-oriented orthonormal frames for the normal bundle to the foliation. If $\Gamma$ is the product of two torsion-free cocompact Fuchsian groups, we have again a Riemannian foliation $\F$ on $M = \Gamma \rs \Hy \times \Hy$, but the lifted foliation on $E_T = \Gamma \rs \Hy \times \PSL$ is not longer minimal (since the leaves closures are parametrised by the compact manifold  $p_2(\Gamma) \rs \PSL$). 
\medskip 

\noindent
(ii) According to the construction given in \cite{BGSS}, let $h$ be an injective group homomorphism  of a torsion-free discrete subgroup $\Gamma$ of $\PSL$ into $SO(3,\R)$ such that $\overline{h(\Gamma)} = SO(3,\R)$. As observed in Examples \ref{examplesLie}.(ii), the horizontal foliation of  $H = \PSL \times SO(3,\R)$ induces a minimal $SO(3)$-Lie foliation on the quotient of $H$ by the image of the injective group homomorphism $i : \Gamma \to \PSL \times SO(3,\R)$ deduced from $h$. Thus $\Gamma$ acts freely and properly on the product $\Hy \times S^2$ and  the quotient manifold $M = \Gamma \rs \Hy \times S^2$ admits a minimal Riemannian foliation $\F$, which can be directly defined by the suspension of the representation of $\Gamma$ as a group of orientation-preserving isometries of $S^2$. As before, the foliation $\F$ lifts to a minimal 
$SO(3,\R)$-Lie foliation $\F_T$ on $E_T = \Gamma \rs \Hy \times SO(3,\R)$ defined by the representation $h : \Gamma \to SO(3,\R)$. 
\medskip 

In both examples, the leaves of $\F$ are dense hyperbolic planes and cylinders. By replacing the unit tangent bundle 
$X = T^1 \F$ by $X_T = T^1 \F_T = \Gamma \rs \PSL \times G$ where $G = \PSL$ or $G = SO(3,\R)$, we can derive $U$-minimality on $X$ from 
$U$-minimality on $X_T$. The same strategy can be applied to general Riemannian foliations by using Molino's theory \cite{Mo1} and more specifically the following important result:
\end{examples}

\begin{namedstheorem}[Molino] If $\F$ is a smooth Riemannian foliation of a compact connected manifold, then $\F$ lifts to a smooth foliation $\F_T$  on the transverse orthonormal frame bundle $E_T$ of $\F$ such that
\begin{list}{\labelitemi}{\leftmargin=5pt}

\item[(i)] the closures of the leaves of $\F_T$ are the fibres of a locally trivial smooth fibration $\pi_T : E_T \to B_T$;

\item[(ii)] there is a Lie group $G$ such that $\F_T$ induces a $G$-Lie foliation with dense leaves  \hspace*{-8pt} on each fibre of $\pi_T$.
\end{list}
\end{namedstheorem}

\noindent
Let us explain how to construct the lifted foliation $\F_T$. Assume $\F$ is given by foliated charts 
$\varphi _ i : U_i \to P_i \times T_i$ from open subsets $U_i$ that covers $M$ to the product of open discs $P_i$ and $T_i$ in $\mathbb{R}^p$ and $\mathbb{R}^q$ respectively. If we can endow each local transversal $T_i$ with a Riemannian metric $g_i$ that is invariant by the changes of chart, the foliation $\F$ is Riemannian. 
From this local point of view, it is clear that each canonical projection $\pi_i = p_2 \scirc \varphi_i : U_i \to T_i$ becomes a Riemannian submersion, so the lifted foliation $\F_T$ is defined by the projection
$\pi_{i_\ast} : E_T |_{U_i} = p_T^{-1}(U_i) \to E_i$ where $p_T : E_T \to M$ is the bundle map and $E_i$ is the orthonormal frame $O(q,\R)$-bundle over $T_i$. By construction, if $\F$ is a foliation by hyperbolic surfaces, then $\F_T$ is also a foliation by hyperbolic surfaces.  As in Examples~\ref{examplesRiemann}, the $U$-minimality problem for Riemannian foliations by hyperbolic surfaces can be reduced to the simpler case of Lie foliations by hyperbolic surfaces:

\begin{proposition} \label{thmriemann} Let $\F$ be a minimal Riemannian foliation by hyperbolic surfaces of a compact connected manifold $M$. 
Let  $X = T^1\F$ and $X_T = T^1 \F_T$ be the unit tangent bundles of $\F$ and $\F_T$. For $F =U$, $B$ or $\PSL$, if the right $F$-action on $X_T$ is minimal, then the right $F$-action on $X$ is minimal. 
\end{proposition}

\begin{proof}  We first see that, under the conditions above, $\F_T$ is a minimal $G$-Lie foliation by hyperbolic surfaces. Let $L_T$ be any fibre of the basic fibration $\pi_T : E_T \to B_T$, and let 
$\F_T |_{\textstyle L_T}$ be the $G$-Lie foliation induced by $\F_T$ on $L_T$. Any closed subset 
$C \subset L_T$ saturated by $\F_T |_{\textstyle L_T}$ is also a closed subset of $E_T$ saturated by $\F_T$. Since $O(q,\R)$ is compact, its image $p_T(C)$ is a closed subset of $M$ saturated by $\F$. Now, since $\F$ minimal, we have $p_T(C) = M$ and hence the fibre $F_T$ projects on the whole manifold $M$. In other words, $F_T = E_T$ and $B_T$ reduces to one point. Thus, according to Molino's theorem,  $\F_T$ is a minimal $G$-Lie foliation. By construction, its unit tangent bundle $X_T = T^1 \F_T$ is a $O(q,\R)$-principal bundle over  $X = T^1 \F$ and the right $\PSL$-action on $X$ is induced by the right $\PSL$-action on $X_T$. Finally, if the right $F$-action on $X_T$ is minimal for $F =U$, $B$ or $\PSL$, then the right $F$-action on $X$ is minimal too. 
\end{proof}

A minimal Riemannian foliation $\F$ is homogeneous if and only if the lifted foliation $\F_T$ is homogeneous. In this case, using Proposition~\ref{thmriemann}, we obtain the following corollary of Theorem~\ref{thm2}: 

\begin{corollary} \label{corthm2}
 Let $X = T^1\F$ be the unit tangent bundle of a minimal Rieman\-nian foliation $\F$ by hyperbolic surfaces of a compact manifold $M$. 
 Assume $\F$ is homogeneous. Then the right $U$-action on $X$ is minimal. 
\end{corollary}

\noindent
In fact, any minimal Riemannian foliation is transversely homogeneous, like the transversely hyperbolic and transversely elliptic foliations described in Examples~\ref{examplesRiemann}, see \cite{G}.  Now it is a natural question to ask if the generalisation of Hedlund's theorem holds for transversely homogeneous foliations.

\section{Proof of Theorem~\ref{thm3}} \label{Section3}

 In this section, we prove that Theorem~\ref{thm2} fails when we consider a transversely homogeneous foliation instead a $G$-Lie foliation. We start by exhibiting a first example of transversely projective counter-example: 
 
\begin{example} \label{firstexample}
Let $\Gamma$ be a torsion-free discrete subgroup of $\PSL$. Consider its diagonal action on  $\PSL \times \partial \Hy$ given by 
$\gamma(f,\xi) = (\gamma f, \gamma(\xi))$ for all $\gamma \in \Gamma$ and for all $(f,\xi) \in \PSL \times \partial \Hy$. 
If $\Gamma$ is cocompact, then $\Gamma$ acts minimally on $\partial \Hy$ and hence the right $\PSL$-action on $X =  \Gamma \rs \PSL \times \partial \Hy$ is minimal. However, the right $B$-action is not minimal because the dual $\Gamma$-action on $\partial \Hy \times \partial \Hy$ is not minimal. More precisely, the diagonal set $\Delta$ consisting of all pairs $(\xi,\xi)$ is a non-trivial $\Gamma$-invariant closed subset of $\partial \Hy \times \partial \Hy$. This means that neither Proposition~\ref{charactB}, nor Proposition~\ref{cocompactB} can be extended to this more general context. In fact $\Delta$ is the unique non-empty $\Gamma$-minimal subset of  $\partial \Hy \times  \partial \Hy$. By duality, 
$$\M = \{ \, \Gamma \Big( \pm\!\matrice{a}{b}{c}{d}, \frac a c \Big) \, / \, \pm\!\matrice{a}{b}{c}{d}\in \PSL \, \}$$
 is the unique non-empty $B$-minimal subset of $X$, proving the first part of Theorem~\ref{thm3}. However, 
we have the following result: 
 
 \begin{proposition} \label{U-minimalset}
 The set $\M$ is the unique non-empty $U$-minimal subset of $X$. 
 \end{proposition} 
 
 \begin{proof} 
 According to Hedlund's theorem, the $U$-action on $\Gamma \rs \PSL$ is minimal. By duality, the $\Gamma$-action on $E = \R^2 - \{0\}/\{\pm Id\}$ is minimal too. This implies that the diagonal $\Gamma$-action on the subset 
 $$
\mathcal{K} = \{ \, (v,\xi) \in E \times \partial \Hy \, / \mbox{ $v$ is collinear to $\vect{\xi}{1}$ if $\xi \neq \infty$ and 
collinear to $\vect{1}{0}$ if $\xi = \infty$ } \} 
 $$
 is minimal. Coming back to $X$ and using again duality, we obtain that $\M$ is $U$-minimal. Let $\gamma_1$ and $\gamma_2$ two hyperbolic isometries in $\Gamma$ generating a Schottky group, that is, the fundamental group of a pair of pants. Since $\Gamma$ acts minimal on $E$, for each point 
 $(v,\xi) \in E \times \partial \Hy$, there exists $(v_1,\xi_1) \in \Gamma(v,\xi)$ such that $\gamma_1 v_1 = \lambda_1 v_1$ with $|v_1| > 1$. 
 Moreover, since $\gamma_1$ and $\gamma_2$ have no common fixed points, for some sequences $\{ p_n \}_{n \geq 0}$ and $\{ q_n \}_{n \geq 0}$ in $\Z$, we have:
 $$
 \lim_{n \to + \infty} \gamma_2^{p_n} \gamma_1^{q_n} v_1 = v_2 \quad
 \mbox{ and } \quad
 \lim_{n \to + \infty} \gamma_2^{p_n} \gamma_1^{q_n} \xi_1 = \xi_2 \vspace{-1ex}
 $$
 where $\gamma_2 v_2 = \lambda_2 v_2$ with $|v_2| > 1$ and $v_2$ is collinear to $\vect{\xi_2}{1}$ or $\vect{1}{0}$.  It follows that 
 $\overline{\Gamma(v,\xi)} \cap \mathcal{K} \neq \emptyset$ and hence $\overline{\Gamma(v,\xi)} = \mathcal{K}$.
 \end{proof}
 
Notice also that the dynamics of the $B$-action on $X - \M$ are related to the dynamics of the geodesic flow on 
$\Gamma \rs \PSL$ since each point $(\xi^-,\xi^+) \in \partial \Hy \times \partial \Hy -\Delta$ represents a geodesic in $\Hy$. 
Like in Examples~\ref{examplesRiemann}, $\Gamma$ acts freely and properly discontinuously on  $\Hy \times  \partial \Hy$, so the horizontal foliation of $\Hy \times  \partial \Hy$ induces a foliation $\F$ on the quotient manifold $M = \Gamma \rs \Hy \times \partial \Hy$ whose unit tangent bundle is $X$. 
\end{example}

\begin{proposition} \label{Bfirstexample}
The transversely homographic foliation $\F$ is defined by a locally free $B$-action whose orbits are dense hyperbolic planes and cylinders.
\end{proposition}

\begin{proof} By construction, since $\partial \Hy$ is identified to homogeneous space $\PSL/B$, $\F$ is a minimal transversely homographic foliation whose leaves are dense hyperbolic planes and cylinders. To prove that they are the orbits of a smooth $B$-action on $M$, we use an idea of 
Mart\'{\i}nez and Verjovsky from \cite{MV}. Indeed, the bundle map $\pi : X = T^1 \F \to M$ becomes a diffeomorphism from the unique $B$-minimal set 
$\M$ onto the quotient manifold
$$M = \{  \, \Gamma \big(\frac{ai+b}{ci+d},\frac a c \big) \,  /   \, \pm\!\matrice{a}{b}{c}{d}\in \PSL \,\}.$$
It follows that $\F$ is defined by a locally free $B$-action, which is conjugated to the $B$-action on $\M$. 
\end{proof} 

We are now interested to provide another counter-example (locally modelled by $\PSL \times \R$) having a non-trivial $B$-minimal set which is not $U$-minimal. Although the construction is classical, see \cite{GS}, we recall some details. Thus, 
any matrix 
$$
A = \matrice{a}{b}{c}{d} \in SL(2,\Z)
$$
defines an orientation-preserving automorphism of the torus $T^2 = \R^2/\Z^2$.
The Lie group automorphism $(z,t) \in T^2 \times \R \mapsto (A(z),t+1) \in T^2 \times \R$
generates a free and properly discontinuous $\Z$-action on the product $T^2 \times \R$. Its orbit space is a compact $3$-manifold $T^3_A$ admitting a natural structure of fibre bundle over $S^1 = \R / \Z$. 
In fact, we consider only the hyperbolic case where $tr \, A >2$ and hence $A$ has two real eigenvalues $\lambda > 1$ and $1/ \lambda < 1$. 

\begin{lemma} \label{eigenvectors}
If $A$ is hyperbolic, then the eigenvectors $u$ and $v$ associated to the eigenvalues $\lambda > 1$ and $1/ \lambda < 1$ generate two different eigenlines with irrational slope.
\end{lemma} 

\begin{proof} Assume on the contrary that $w =(p,q)$ is an eigenvector of $A$ where $p,q \in Z$ are relatively prime (including the cases where $p=0$ and $q=1$ or $p=1$ and $q=0$). Then there exists $w' = (p',q') \in \Z$ such that $pq' - qp' = 1$ and then the matrix
$$
B = \matrice{p}{p'}{q}{q'} \in SL(2,\Z)
$$
satisfies $Be_1= w$ and $Be_2 = w'$, This implies that $e_1$ is an eigenvector of $B^{-1}AB$ so that $B^{-1}AB$ is an upper triangular matrix. Since $B^{-1}AB$ belongs to $SL(2,\Z)$, we have $\pm B^{-1}AB \in U$. Then the eigenvalues of $A$ are equal to $\lambda = \pm 1$, which contradicts the hyperbolicity of $A$. 
\end{proof}

\begin{example} \label{secondexample}
The foliation of $\R^2$ by parallel $u$-lines induces a minimal flow on $T^2$. The product of this foliation with the vertical factor defines a $2$-dimension foliation of $T^2 \times \R$ which is invariant by the $\Z$-action described above. Thus, by passing to the quotient, we obtain a foliation $\F$ of $T^3_A$ whose leaves are planes and cylinders. Indeed, according to \cite{BR} and denoting by $\pi$ the projection from $\R^2$ onto $T^2$, for each point $(x,y) \in \Q^2$, there is a positive integer $p \geq 1$ such that $A^p\pi(x,y) =  \pi(x,y)$, and so $\F$ contains infinitely many cylindrical leaves. We will see that all the leaves are hyperbolic surfaces. 
Denote by $\A$ the group of orientation-preserving affine transformations of $\R$, which is isomorphic to $B$. Let $\Gamma$ be the discrete subgroup of $\SL \times \A$ generated by 
 \begin{eqnarray}
T_1(x,y,t) & = &( x+1,y,t) \label{t1} \\
T_2(x,y,t) & = & (x,y+1,t)  \label{t2} \\
h_A(x,y,t) & = & (A \vect{x}{y}, t+1). \label{hA}
\end{eqnarray}
acting on $\R^3 = \R^2 \times \R$. The foliated manifold $T^3_A$ is the quotient of $\R^3$ by the action of $\Gamma$, so is endowed with a complete affine structure. Let $u$ and $v$ be the eigenvectors of $A$ verifying $Au = \lambda u$ and $Av = \lambda^{-1} v$. Assume $det(u|v) = 1$. By changing the canonical affine frame $(0,e_1,e_2,e_3)$ by $(0,u,v,e_3)$ in $\R^3$, the transformations (\ref{t1}), (\ref{t2}) and 
(\ref{hA}) can be written as follows: 
\begin{eqnarray}
T_1(x',y',t') & = &( x'+a', y'+b',t')  \label{T1} 
\\
T_2(x',y',t') & = & (x'+c',y'+d',t')  \label{T2} \\
h_A(x',y',t') & = & (\lambda x',  \lambda^{-1}y', t'+1) \label{HA}
\end{eqnarray}
where $u = (d',-b')$ and $v = (c',-a')$. Thus, from Lemma~\ref{eigenvectors}, the entries $a'$ and $c'$, and the entries $b'$ and $d'$ are linearly independent over $\Z$. In fact, the universal covering $\R^3$ of 
$T^3_A$ can be identified with the product $\Hy \times \R$ by sending each point $(x',y',t') \in \R^3$ to 
the point $(z',y') = (x' + i \lambda^{t'},y') \in \Hy \times \R$. In this model, the transformations (\ref{T1}), (\ref{T2}) and 
(\ref{HA})  can be written
\begin{eqnarray}
T_{1\ast}(z',y') & = &( z'+a', y'+b')  \label{T*1} 
\\
T_{2\ast}(z',y') & = & (z'+c',y'+d')  \label{T*2} \\
h_{A\ast} (z',y') & = & (\lambda z',  \lambda^{-1}y') \label{H*A}
\end{eqnarray}
Moreover, the foliation $\F$ lifts to the horizontal foliation of $\Hy \times \R$.  
\end{example}

\begin{proposition} [\cite{GS}] \label{GS}
 The transversely affine foliation $\F$ is defined by a locally free $B$-action
whose orbits are dense hyperbolic planes and cylinders. \end{proposition}

\begin{proof} Firstly, we remark that $\F$ admits a affine transverse structure because the $\Gamma$-action on  $\Hy \times \R$  
defined by (\ref{T*1}), (\ref{T*2}) and (\ref{H*A}) induces an affine action on the 
$\R$-factor generated by 
\begin{eqnarray}
\overline{T}_{1\ast}(y' )& = & y'+b' \label{T'1} 
\\
\overline{T}_{2\ast}(y') & = & y' + d' \label{T'2} \\
\overline{h}_{A\ast} (y') & = & \lambda^{-1}y'  \label{H'A}
\end{eqnarray}
We know that the leaves of $\F$ are planes or cylinders. Since $b'$ and $d'$ are linear independent over $\Z$, $\overline{T}_{1\ast}$ and 
$\overline{T}_{2\ast}$ generate a dense subgroup of translations of $\R$ and hence all leaves are dense.
On the other hand, there is a natural right $B$-action on $\Hy \times \R$ where each element 
$$
\matrice{\sqrt \alpha}{\beta / \sqrt \alpha}{0}{1 / \sqrt \alpha} 
$$
of $B$ acts by homographies on the first factor $\Hy$ sending $z$ to $\alpha z + \beta$,  and trivially on the second factor $\R$. Since this free $B$-action commutes with the $\Gamma$-action, it induces a locally free $B$-action on $T^3_A$ whose orbits are just the leaves of $\F$. 
\end{proof}

\begin{remark} The group law
$(x',y',t')(x'',y'',t'') = (x'+\lambda^{t'}x'', y' +\lambda^{-t'}y'', t'+t'')$ defines a group structure on $\R^3$ that becomes a Lie group isomorphic to the solvable Lie group $Sol^3$. Each horizontal leaf $\Hy \times \{y'\}$ is the orbit of any point $(x',y',t')$ by the right $B$-action determined by the inclusion
$$
i\big( \matrice{\sqrt \alpha}{\beta / \sqrt \alpha}{0}{1 / \sqrt \alpha} \big) = (\beta,0,\frac{log \, \alpha}{log \,  \lambda})
$$
of $B$ as closed subgroup of $Sol^3$. The orbits of the corresponding $U$-action are the horizontal $x'$-lines in $\R^3$, which correspond to the parallel $u$-lines before changing the affine frames. 
\end{remark}

\begin{proof}[Proof of Theorem~\ref{thm3}] 
By construction, the unit tangent bundle $X = T^1\F$ is the quotient of $T^1 \Hy \times \R$ by the $\Gamma$-action generated by  (\ref{T*1}), (\ref{T*2}), and (\ref{H*A}).  By duality,  the right $B$-action on $X = T^1\F$ has the same dynamics as the $\Gamma$-action on 
$\partial \Hy \times \R$ generated by the transformations 
\begin{eqnarray}
T_{1\ast} (\xi,y') & = &( \xi+a', y'+b')  \label{T1*} 
\\
T_{2\ast} (\xi,y') & = & (\xi+c',y'+d')  \label{T2*} \\
h_{A\ast} (\xi,y') & = & (\lambda \xi,  \lambda^{-1}y') \label{HA*}
\end{eqnarray}
We first observe that $\{\infty\} \times \R$ is a closed $\Gamma$-invariant subset of $\partial \Hy \times \R$. Minimality and uniqueness arise from 
$\lim_{n \to +\infty} h_{A\ast}^n (\xi,y') = (\infty,0)$ for all $(\xi,y') \in (\partial \Hy -\{0\}) \times \R$ and
$\overline{\Gamma(0,y')} = \partial \Hy \times \R$ for all $y' \in \R$. Therefore, there is an unique minimal set $\M$ for the right $B$-action on 
$X = T^1\F$, obtained as the $\Gamma$-quotient of the pre-image of $\{\infty\} \times \R$ by the canonical projection of $T^1 \Hy \times \R$ onto $\partial \Hy \times \R$.  But the closure of each $U$-orbit reduces to a toroidal fibre of the bundle structure of $T^3_A$ over $S^1$. 
\end{proof}

An important difference between this example and all the previous ones is that the discrete subgroup $\Gamma$ of $\PSL \times \A$ projects onto a subgroup $p_1(\Gamma)$ of $\PSL$ which is neither discrete, nor dense. Moreover, by construction, the $\Gamma$-action induced on 
$(\partial \Hy - \{\infty\} ) \times \R$ is conjugated to the action of the group of affine transformations of  $\R^2$ generated by the linear automorphism $A$ and the translations $t_1(x,y) = (x+1,y)$ and $t_2(x,y) = (x,y+1)$. Thus, the $B$-action induced on $X - \M$ is dual to the $\Z$-action on $T^2$ generated by $A$, whose topological dynamics have been carefully described by R. Adler \cite{A}. 
It follows that there are $B$-orbits which are dense in $X$, and others whose closures are not manifolds. 

\section{Final comments} \label{SFC}

As we already mentioned, Example~\ref{firstexample} shows that neither Proposition~\ref{charactB}, nor Proposition~\ref{cocompactB} are valid in the non-Riemannian case. Nevertheless, even if Hedlund's Theorem cannot be generalised, the question formulated by Mart\'{\i}nez and Verjovsky remains open: {\em is it true that the horocycle flow on the unit tangent bundle $X=T^1 \F$ of a minimal foliation $\F$ of a compact manifold $M$ by hyperbolic surfaces is minimal if and only the $B$-action is minimal?} Example~\ref{secondexample} proves that this conjecture cannot be strengthened by establishing an equivalence between $U$-minimal and $B$-minimal sets. 
As proved in  Propositions~\ref{Bfirstexample}~and~\ref{GS}, Examples~\ref{firstexample}~and~ \ref{secondexample} are defined by locally free $B$-actions.
\medskip 


In \cite{MV}, Mart\'{\i}nez and Verjovsky have reformulated their conjecture as follows: {\em is it true that for any compact manifold foliated  by dense hyperbolic surfaces, either the foliation is defined by a $B$-action or the $U$-action on $X$ is minimal?}  Notice that the non-homogeneous $\PSL$-Lie foliation constructed in \cite{HMM} (as well as any Riemannian foliation) cannot be defined by a $B$-action, since it admits a transverse invariant volume, according to Proposition~3.1 of \cite{P}. It is an open question to know if the $U$-action is minimal or not. 
\medskip 

Progress on this issue is interesting but very restricted in the non-homogeneous case. We place in an appendix some results, which are related to those of Sections~\ref{Section1} in this more general context. 

\section*{Appendix: $U$-minimality for some non-homogeneous foliations}

Let us introduce now the group of orientation-preserving $C^r$-diffeomorphisms $\Diff$ of some orientable $C^r$-manifold $F$, $0 \leq r \leq +\infty$ or $r = \omega$, and give some remarks for the case where $M$ is a compact manifold obtained as the quotient of 
$\Hy \times F$ by a subgroup $\Gamma \subset \PSL \times  \Diff$ acting freely and properly discontinuously on $\Hy \times F$.  Like in Section~\ref{Section1}, we denote by $p_1$ and $p_2$ the first and second projection of  $\PSL \times  \Diff$ onto $\PSL$ and 
$\Diff$ respectively. Recall that $M$ admits a foliation $\F$ induced by the horizontal foliation of $\Hy \times F$ and $\F$ is minimal if and only if  $p_2(\Gamma)$ acts minimally on $F$. Denote by $p$ and $q$ the canonical projections
$$p : \PSL \times F \to \partial \Hy \times F = \PSL / B \times F$$ and $$q : \PSL \times F \to X = \Gamma \rs  \PSL \times F$$
corresponding to the natural right $B$-action and left $\Gamma$-action on $\PSL \times F$.

The first result generalises Example~\ref{secondexample}: 

\begin{aproposition} \label{prop5.3} 
If $p_1(\Gamma)$ is solvable, then the natural right $B$-action on $X$ is not minimal. More precisely, there is a $B$-minimal set
homeomorphic to $M$. 
\end{aproposition}

\begin{proof} Since $p_1(\Gamma)$ is solvable, but it is not included in $PSO(2,\R)$ by the compactness of $M$,  this group fixes a point $\xi \in \partial \Hy$. Then $Z = \{\xi\} \times F$ is a $B$-minimal closed subset of $\partial \Hy \times F$ because $p_2(\Gamma)$ acts minimally on $F$. It follows that $\widetilde{Y} = p^{-1}(Z)$ is a $\Gamma$-invariant and $B$-invariant closed subset of $\PSL \times F$, which is homeomorphic to $\Hy \times F$. Clearly,  we deduce that $\widetilde{Y}$ projects onto a $B$-minimal closed set $Y = q(\widetilde{Y}) \subset X$, which is homeomorphic to $M$. 
\end{proof}

Suppose now that $p_1(\Gamma)$ is not solvable, so $p_1(\Gamma)$ is discrete cocompact or dense. In particular, its action on $\partial \Hy$ is minimal. Assuming $F$ compact and $\Gamma$ torsion-free, we have the following result:

\begin{aproposition} \label{prop5.5} Assume that $F$ is compact, $\Gamma$ is torsion-free, 
and $p_1(\Gamma)$ is not solvable. If $p_2$ is not injective, then 
the natural right $U$-action on $X$ is minimal. 
\end{aproposition}

\begin{proof} Assume the projection $p_2$ is not injective. Since the kernel $N$ is normalised by $p_1(\Gamma)$, the group $p_1(\Gamma)$ is discrete cocompact and $N$ is not cyclic. It follows that $\F$ admits leaves which are not homeomorphic to the plane or
the cylinder. Moreover, the foliated manifold $M$ is the quotient of $\Hy \times F$ by the diagonal action 
$\gamma(z,y) = (\gamma_1(z),\gamma_2(y))$ where $\gamma_1$ is an element of the cocompact discrete subgroup $p_1(\Gamma)$ of $\PSL$ and $\gamma_2$ is the corresponding element of 
$\Diff$. So $\F$ is obtained as suspension of the representation $h : \gamma_1 \in p_1(\Gamma) \mapsto \gamma_2 \in \Diff$. 
Now, let us prove that the corresponding $U$-action is minimal. 
By duality, it is enough to prove that $\Gamma$ acts minimally on the product  $E \times F$ where $E = \R^2- \{0\} / \{ \pm Id\}$. 
Let $v$ be an element of $E$ such that $\gamma_1 v = \lambda_1 v$ for some $\gamma_1 \in N$ with $|\lambda_1| \neq 1$. Since $N$ acts minimally on $\partial \Hy$, it is known \cite{DL} that 
$\overline{Nv} = E$. It follows $\overline{\Gamma(v,y)}$ contains $E \times \{y\}$ for all $y \in F$. 
Using the minimality of the action of $p_2(\Gamma)$ on $F$, we deduce that $\overline{\Gamma(v,y)} = E \times F$ for all $y \in F$. 
Indeed, for each point $(w,z) \in E \times F$, 
there is a sequence  $\{ \gamma_n \}_{n \geq 0}  =  \{(\gamma_{1n},\gamma_{2n})\}_{n \geq 0}$ in $\Gamma$ such that $z = \lim_{n \to +\infty} \gamma_{2n}(y)$. Since 
$(\gamma_{1n}^{-1}w, y) \in E \times \{ y \} \subset \overline{\Gamma(v,y)}$,
we have:
$$
(w,z) = \lim_{n \to +\infty} (w, \gamma_{2n}(y)) = \lim_{n \to +\infty} \gamma_n(\gamma_{1n}^{-1}w, y)  \in \overline{\Gamma(v,y)}.
$$
Finally, since $p_1(\Gamma)$ is discrete cocompact, given any point $(w,z) \in E \times F$, there is another sequence 
$\{ \gamma'_n \}_{n \geq 0}  =  \{(\gamma'_{1n},\gamma'_{2n})\}_{n \geq 0}$ in $\Gamma$ such that 
$v = \lim_{n \to +\infty} \gamma'_{1n}w$. By compactness of $F$, extracting a subsequence if necessary, we may assume that  $\gamma'_{2n}(y)$ converges to a point $y' \in F$.Thus 
$$
\lim_{n \to +\infty}\gamma'_n(w,y) = \lim_{n \to +\infty} ((\gamma'_{1n}w, \gamma'_{2n}(y)) = (v,y')
$$
Since $\overline{\Gamma(v,y')} \subset \overline{\Gamma(w,y)}$ and $\overline{\Gamma(v,y')} = E \times F$, we obtain that $\overline{\Gamma(w,y)} = E \times F$. 
\end{proof} 

\subsection*{Note added in proof} A Hedlund's theorem for foliations by hyperbolic surfaces which admit a leaf that contains an essential loop without holonomy has been announced by the authors in collaboration with M. Mart\'{\i}nez and A. Verjovsky \cite{ADMV} after the submission of this paper.

\end{document}